\documentclass[reqno]{amsart}

\usepackage{amsmath, amssymb, mathtools, mathrsfs, bbm, yhmath}
\usepackage[inline]{enumitem}
\usepackage[hmargin=1.3in, vmargin=1in]{geometry}
\usepackage{xcolor}
\usepackage[pdfencoding=auto, psdextra]{hyperref}
\hypersetup{colorlinks=true, allcolors=[rgb]{0,0,0.6}}
\usepackage{graphicx, tikz, tikz-cd}
\usepackage[all,cmtip]{xy}
\usepackage{appendix}

\numberwithin{equation}{section}

\newcommand{\dd}{\mathrm{d}}
\newcommand{\e}{\mathrm{e}}
\newcommand{\ii}{\mathrm{i}}
\newcommand{\eps}{\varepsilon}
\newcommand{\Id}{\mathbbm{1}}

\newcommand{\N}{\mathbb{N}}

\newcommand{\R}{\mathbb{R}}
\newcommand{\T}{\mathbb{T}}
\newcommand{\Z}{\mathbb{Z}}

\theoremstyle{plain} 
\newtheorem{theorem}{Theorem}[section]
\newtheorem*{theorem*}{Theorem}
\newtheorem{lemma}[theorem]{Lemma}
\newtheorem*{lemma*}{Lemma}
\newtheorem{corollary}[theorem]{Corollary}
\newtheorem*{corollary*}{Corollary}
\newtheorem{proposition}[theorem]{Proposition}
\newtheorem*{proposition*}{Proposition}

\newtheorem*{conjecture*}{Conjecture}

\theoremstyle{remark}

\theoremstyle{definition} 

\newtheorem*{definition*}{Definition}

\newtheorem*{example*}{Example}

\newtheorem*{assumption*}{Assumption}
\newtheorem{remark}[theorem]{Remark}
\newtheorem*{remark*}{Remark}

\title[Periodic Concentrated NLS]{Well-posedness of the periodic nonlinear Schr\"odinger equation with concentrated nonlinearity}

\author{Jinyeop Lee}
\address[J. Lee]{Department of Applied Mathematics, Kyung Hee University, 1732 Deogyeong-daero, Giheung-gu, Yongin-si, Gyeonggi-do, South Korea}
\email{jinyeop.lee@khu.ac.kr}

\author{Andrew Rout}
\address[A. Rout]{Dipartimento di Matematica, Politecnico di Milano, P.zza Leonardo da Vinci 32, 20133, Milano, Italy.}
\email{andrewjames.rout@polimi.it}

\begin{document}

\begin{abstract}
We study the solution theory of the nonlinear Schrödinger equation with a concentrated nonlinearity on the torus. In particular, we establish existence and uniqueness of global energy-conserving solutions for initial data in $H^1$. We also prove local well-posedness of the concentrated nonlinear Schr\"odinger equation below the energy space but above the endpoint of Sobolev embedding theorem. Our methods are based on compactness results and exploiting the Volterra integral equation structure of the problem. To our knowledge, this is the first rigorous solution theory for the concentrated nonlinear Schr\"odinger equation on the circle.
\end{abstract}

\maketitle

\section{Introduction}
The nonlinear Schrödinger equation (NLS) is an important equation in a number areas of physics, serving as an effective equation for a microscopic system on a macroscopic scale and is used in quantum optics. In this paper, we consider the case of an NLS whose nonlinearity affects the evolution at only a single point --- i.e. an NLS with a {\it concentrated nonlinearity}. Such an equation in one spatial dimension can be written as
\[
\ii \partial_t u + \Delta u =  \delta |u|^{2}u\,.
\]
This equation can be used to model a number of different physical systems, such as when a group of electrons experience tunnelling through a double-well potential, or for modelling an impurity in a medium. For physical applications of the equation, we direct the reader to the introduction of \cite{Ten23}, and the references within. Linear equations with point interactions, such as
\begin{equation}
\label{linear_delta}
    \ii \partial_t u + \Delta u =  \alpha \delta  u\,.
\end{equation}
are rigorously constructed via the theory of {\it self-adjoint extensions of symmetric operators}, see for example \cite{AGHK88,FT00}. These equations were first extended to nonlinear interactions in the work of \cite{AT01}.

The {\it concentrated nonlinear Schrödinger equation} in one spatial dimension is given by
\begin{equation}
\label{concentrated_NLS_introduction}
\left\{\begin{alignedat}{2}
\ii \partial_t u + \Delta u &=  \delta |u|^{2}u\,, \\
u(x,0) &= u_0(x) \in H^s(X)\,,
\end{alignedat}\right.
\tag{cNLS}
\end{equation}
where $\delta$ denotes the Dirac delta function.
We note that for a continuous function $f$, one can interpret the product of $f$ with $\delta$ as $\delta f(0)$. In particular, the product is well defined for $f \in H^s(\T)$, for $s > \frac{1}{2}$. Since all the solutions we will consider in this paper to the \eqref{concentrated_NLS_introduction} will be continuous in space, we can always define the product $\delta |u|^2u$.

The \eqref{concentrated_NLS_introduction} formally has two conserved quantities, namely the mass and energy, respectively defined as
\begin{align*}
M(u) &:= \int_{X} |u(x,t)|^2 \, \dd x \,,\\
E(u) &:= \int_{X} |\nabla u(x,t)|^2 \, \dd x + \frac{1}{2} |u(0,t)|^4 \,.
\end{align*}

The \eqref{concentrated_NLS_introduction} has been heavily studied in the case where $X=\R^d$. Indeed, the well-posedness of \eqref{concentrated_NLS_introduction} in one dimension was proved in \cite{AT01}, which was extended to three dimensions in \cite{ADFT03}, before being proved in two dimensions in \cite{CCT19}. However, to the best of the authors' knowledge, there is no solution theory for the \eqref{concentrated_NLS_introduction} in the setting that $X = \T$. The first key result of this paper is the following theorem, which establishes well-posedness of the \eqref{concentrated_NLS_introduction} for initial conditions in the energy space.
\begin{theorem}
\label{existence_uniqueness_theorem}
Suppose that $u_0 \in H^1(\T)$. Then, for any $T > 0$, there is a unique function $u \in C([0,T]; H^1(\T))$ such that $u$ solves \eqref{concentrated_NLS_introduction}. Moreover, one has
\begin{align*}
M(u(t)) = M(u_0) \quad\text{and}\quad E(u(t)) = E(u_0)
\end{align*}
for any $t \in [0,T]$. Moreover the $C([0,T];H^1(\T))$ norm of the solution depends continuously on the initial data.
\end{theorem}

\medskip

\begin{remark}
In the statement of Theorem \ref{existence_uniqueness_theorem}, and indeed throughout the paper, when we say a solution, we mean a mild solution. In other words, $u$ satisfies the integral equation
\begin{equation*}
u(t) = S(t)u_0 - \ii\int_0^t S(t-t') \delta |u|^2u \, \dd t' \,.
\end{equation*}
We direct the reader to \eqref{Schro_kernel_defn} for the definition of the Schr\"odinger kernel $S(t)$.
\end{remark}

\begin{remark}
We make the following remarks about Theorem \ref{existence_uniqueness_theorem}. 
\begin{enumerate}
\item For convenience, we place the $\delta$ function at the origin, but our analysis readily extends to a $\delta$ function located at any other point on $\T$. Furthermore, it easily generalizes to any finite linear combination of $\delta$ functions at arbitrary points on $\T$, each with a positive real coefficient.
\item Our results are stated for the cubic \eqref{concentrated_NLS_introduction}, but the analysis easily carries over to any power like nonlinearity of the form $\delta|u|^{2p}u$, since we work above the endpoint of the Sobolev embedding theorem.
\item The choice of a defocusing nonlinearity is based on not wanting to consider the case of finite time blow up in any of the approximations to the \eqref{concentrated_NLS_introduction} which we will consider throughout the paper. An interesting future direction of research would be to consider the problem of blow-up solutions in the setting of the \eqref{concentrated_NLS_introduction} with a focusing nonlinearity given by $-\delta |u|^{2p}u$. We leave this to future work.
\end{enumerate}
\end{remark}

\begin{remark}
Let us briefly remark on the difficulty of working on the torus $\T$ versus working in free space $\R$. When working in free space, a central object of study is the following integral kernel
\begin{equation*}
S^{\delta}_{\R}(t) = \frac{1}{\sqrt{-\ii t}} \e^{\frac{\ii a^2}{t}}.
\end{equation*}
The $\sqrt{t}$ means that the kernel has smoothing properties when convolving in time, similarly to the standard Strichartz estimates. These smoothing properties were proved in \cite[Lemma 3]{AT01}. One also computes that this kernel lies in $H^{-\eps}_{\mathrm{loc}}(\R)$ for any $\eps > 0$. However, when working on the torus, one needs to analyse the following kernel
\begin{equation}
\label{S_delta_definition_introduction}
S^{\delta}(t) = \sum_{n \in \Z} \e^{-\ii n^2 t}.
\end{equation}
This no longer has the smoothing properties of the kernel in $\R$, so the Volterra integral equation arguments from \cite{AT01} do not easily generalise. Further, as a distribution one computes that it lies in $H^{-\frac{1}{4}-\eps}_{\mathrm{\mathrm{loc}}}(\R)$. Moreover, formally applying the Poisson summation formula to the periodic kernel, one computes 
\begin{equation*}
S^{\delta}(t) = \sqrt{\frac{\pi}{\ii t}} \sum_{n \in \Z} \e^{\frac{\ii n^2}{4 t}}.
\end{equation*}
Up to rescaling, we can interpret this problem as analogous to the case on $\R$ with a $\delta$ function placed at each integer point --- forming a lattice of impurities. Similarly, on the torus, solutions of \eqref{concentrated_NLS_introduction} are forced to repeatedly interact with the impurity as they wrap around the torus.
\end{remark}

We also prove the following result about local solutions for initial conditions in $H^s(\T)$, $s \in (\frac{1}{2},1)$.
\begin{theorem}[Existence of unique mass-conserving solutions below $H^1$]
\label{LWP_theorem}
Suppose that $u_0 \in H^{s}(\T)$, for $s \in (\frac{1}{2},1)$. Then there is some $T = T(\|u_0\|_{H^s(\T)}) > 0$ and a unique function $u \in C([0,T];H^s(\T))$ satisfying
\begin{equation*}
u(t) = S(t)u_0 - \ii\int_0^t S(t-t') \delta |u|^2u(t') \, \dd t' \,.
\end{equation*}
Moreover, one has
\begin{equation*}
M(u(t)) = M(u_0),
\end{equation*}
and the solution depends continuously on the initial condition $u_0$ in ${C}([0,T];H^s(\T))$ norm.
\end{theorem}

\begin{remark}
Let us remark that the proof of Theorem \ref{LWP_theorem} is independent of the sign of the nonlinearity, so also holds for the focusing problem. By standard Volterra integral equation theory, there is a blow-up criterion, see Remark \ref{blow-up_remark}.
\end{remark}

\begin{remark}
\label{higher_dimensional_remark}
Let us remark that the problem in two and three dimensions is significantly more complicated. In one dimension, the domain of the self-adjoint extension of the Hamiltonian used to rigorously construct \eqref{concentrated_NLS_introduction} is given by $H^1(\R)$. However, in two and three dimensions, it is a space which is strictly larger than $H^1(\R^d)$. We direct the reader to for example the introduction of the review paper \cite{Ten23} for full details about the difference between one, two, and three dimensions, and for a precise statement of the correct range of initial conditions to consider. Since we focus on the one dimensional problem in this paper, we do not comment further on the higher dimensional cases in this paper.
\end{remark}

\subsection{Method of Proof}
The problems of well-posedness in the energy and below the energy spaces are considered separately. In both cases, we show that there is a unique solution to the charge equation
\begin{equation}
\label{charge_equation_introduction}
q(t) = S(t) u_0(0) - \ii \int_0^t S^\delta(t-t') |q|^2q(t') \, \dd t',
\end{equation}
where we recall \eqref{S_delta_definition_introduction}. This is done by analysing the kernel $S^\delta$ and proving the existence of a fixed point. In both the local and global cases, this suffices to prove the existence of solutions.

However, to prove uniqueness of solutions, we need to prove that we have persistence of regularity (or sufficient regularity to define the product of $|u|^2u$ with $\delta$). Otherwise, there could be a solution which is not completely determined by its value at zero.

For initial data in the energy space, we proceed via compactness methods. In particular, we consider the following equation
\begin{equation}
\label{smoothed_NLS_introduction}
\left\{\begin{alignedat}{2}
\ii \partial_t u^\eps + \Delta u^\eps &=  V^{\eps} |u^\eps|^{2}u^\eps, \\
u^\eps(x,0) &= u_0(x) \in H^1(\T) \, ,
\end{alignedat}\right.
\tag{sNLS}
\end{equation}
where $V^{\eps}\rightharpoonup\delta$ as $\eps\to0$.
Since the initial data is in the energy space, and the problem is defocusing, we can obtain global boundedness on the $H^1(\T)$ of $u^\eps(t)$. In particular, we then use a compactness argument to show that there must be a solution in $C_t([0,T]) H^1_x(\T)$. Since such a solution is determined by its value at zero, we have uniqueness.

For initial conditions in $H^s(\T)$, $s \in (\frac{1}{2},1)$, we no longer have uniform bounds for solutions of \eqref{smoothed_NLS_introduction}. We thus adapt an argument from \cite{AT01} to show that the $H^s(\T)$ norm of $u(t)$ is controlled by the Sobolev norm of $q$. Conservation of mass is proved via a density argument and using the conservation of mass for initial data in $H^1(\T)$.

\subsection{Previously Known Results}
\label{Section:previous}

We briefly review the previously known results in the setting $X = \R$. For $X = \R$, the solution theory of \eqref{concentrated_NLS_introduction} was first considered in \cite{AT01}, where the authors showed the existence and uniqueness of solutions for initial conditions in $H^{s}(\R)$ for $s > \frac{1}{2}$. They also prove the conservation of mass and energy for sufficiently regular initial conditions, and they show that for extremely attractive potentials, one can have blow up of solutions. These results were achieved by treating the problem as a Volterra integral equation, and analysing the time smoothing properties of the Schrödinger kernel in this setting. An independent proof of well-posedness in the energy space based on the dispersive estimates from the Schrödinger kernel was given in \cite{KK07}. Bound states and the orbital stability of the \eqref{concentrated_NLS_introduction} were considered in \cite{BD21}.

It was also shown in \cite{CFNT14} that solutions of the one dimensional \eqref{concentrated_NLS_introduction} can be constructed as the limit of the smoothed NLS as the inhomogeneity converges to the $\delta$ function. This result was central to the first author's work with Adami \cite{AL25}, in which they derived the \eqref{concentrated_NLS_introduction} as the miscroscopic limit of a many-body Schrödinger equation. 

In \cite{HGKV24} the authors consider the solution theory of an NLS with more singular interactions than the $\delta$ function on $X=\R$. In particular, the authors prove the existence and uniqueness of solutions to the equation with $u_0 \in H^1$. They recently extended this to well-posedness and scattering for initial conditions in $L^2(\R)$ in \cite{HGKV25}, and they obtained the concentrated limit for such initial conditions.

In $X = \R^2$ and $\R^3$, the well-posedness of the \eqref{concentrated_NLS_introduction} was proved in \cite{CCT19} and \cite{ADFT03} respectively. Further results in two and three dimensions can be found in \cite{ACCT20, ACCT21} and \cite{ABCT22,DTT24} respectively. For a full summary of the known results for the \eqref{concentrated_NLS_introduction} in $\R,\R^2,$ and $\R^3$, we direct the reader to the review \cite{Ten23}, and the references within. The well-posedness of \eqref{concentrated_NLS_introduction} on the half-line was also considered in \cite{HKO25,HOS25}.

This is, as far as the authors are aware, the first result on the existence and uniqueness of solutions for the \eqref{concentrated_NLS_introduction} on a periodic domain. We mention the work \cite{FT00}, which analyses the self-adjoint extensions for the linear problem, \eqref{linear_delta}, on $\T$.

\subsection{Outline of the Paper}
We briefly outline the structure of the paper. In Section \ref{Section:background}, we fix notation and recall background results. Section \ref{Section:sNLS} establishes the existence of energy-conserving solutions for the smoothed NLS. We then show uniqueness by proving that there is a unique solution to the charge equation \eqref{charge_equation_introduction}. In Section \ref{Section:LWP}, we consider the problem of local well-posedness below the energy space. Section \ref{Section:open} briefly mentions some problems left open by our work. Finally, Appendix \ref{SEC:proposition} contains a proof of a technical proposition needed in our proof.

\section{Notation and Preliminaries}
\label{Section:background}

\subsection{Notation}
Throughout this paper, we will denote by $C > 0$ a positive constant which can change line-to-line. 
To indicate that $C$ depends on the parameters ${x_1, \ldots, x_n}$, we write $C = C(x_1, \ldots, x_n)$. 
We will write $a \lesssim b$ for $a \leq C b$, and if $C(x_1,\ldots, x_p)$, we write $a \lesssim_{x_1,\ldots,x_p} b$. We denote $a\sim b$ if $a\lesssim b$ and $b\lesssim a$.

We denote by $\Id_A$ the indicator function $\Id_A: \R \to \{0,1\}$ defined by
\begin{equation*}
\Id_A(t) := 
\begin{cases}
    1 \quad \text{ if } t \in A \,, \\
    0 \quad \text{ if } t \notin A \,.
\end{cases}
\end{equation*}
We use the notation $c^+$ and $c^-$ to represent $c + \eta$ and $c - \eta$, respectively, for a small constant $\eta > 0$. We also adopt the notation of $u(t) := u(x,t)$, and when considering the function at $0$ in space, we will write $u(0,t)$. 

\subsubsection*{Fourier transforms and Sobolev spaces}
For a function $f \in L^1(\mathbb{T})$, we will denote its {\it Fourier coefficients} by
\begin{equation*}
\hat{f}(n) \sim \int_{\mathbb{T}} f(x) \e^{- \ii n x} \, \dd x \,,
\end{equation*}
where $n \in \Z$. Where clear, we will omit the domain of integration. For $s\in \R$, we also define the $H^s(\T)$ norm by
\begin{equation*}
\|f\|_{H^{s}_x(\T)}^2 \sim \sum_{n \in \Z} (1+n^2)^{s} |\hat{f}(n)|^2.
\end{equation*}
We will also write
\[
\hat{f}(\eta) \sim \int_{-\infty}^\infty f(t)\e^{- \ii t\eta}\,\dd t
\]
for the {\it Fourier transform in time}. It will be clear from context whether we are performing Fourier transforms in time or space. Moreover, we define the Sobolev norm of $F: \R \to \R$ as
\begin{equation*}
\|F\|_{H^{b}_t(\R)} := \|(1+\eta^2)^\frac{b}{2}\hat{f}(\eta)\|_{L^2_{\eta}} \,.
\end{equation*}

\subsubsection*{Schrödinger kernel}
The {\it periodic Schrödinger semigroup} acts on $f$ via
\begin{equation}
\label{Schro_kernel_defn}
S(t)f(x) \sim \sum_{n \in \mathbb{Z}} \hat{f}(n) \e^{-\ii n^2 t}\e^{\ii n x} 
\end{equation}

\subsection{Preliminary Results}
We first record the following lemma about the regularity of the characteristic function.
\begin{lemma}
\label{regularity_cut-off_lemma}
Suppose $I \subset \R$ is a compact interval. Then $\Id_I \in H^{s}_t(\R)$ for any choice of $0< s < \frac{1}{2}$. Moreover $\|\Id_I\|_{H^{s}} \to 0$ as $|I| \to 0$.
\end{lemma}
\begin{proof}
Let $0<s<\tfrac12$ and, without loss of generality, set $I=[0,a]$ with $a>0$. Using the Gagliardo--Slobodecki\u{\i}--Sobolev characterization of $H^s(\R)$ for $s\in(0,1)$,
\[
\|\Id_I\|_{H^{s}(\R)}^{2} \sim \|\Id_I\|_{L^2(\R)}^{2}
  + \int_{\R}\!\int_{\R}\frac{|\Id_I(x)-\Id_I(y)|^{2}}{|x-y|^{1+2s}}\,\dd y\,\dd x\,.
\]
Clearly $\|\Id_I\|_{L^2(\R)}^{2}=a\to0$ as $a\to0$. For the double integral we note that
$|\Id_I(x)-\Id_I(y)|=1$ if and only if exactly one of $x,y$ lies in $I$. Hence
\[
\int_{\R}\!\int_{\R}\frac{|\Id_I(x)-\Id_I(y)|^{2}}{|x-y|^{1+2s}}\,\dd y\,\dd x
=2\int_0^a\!\left(\int_{(-\infty,0]\cup[a,\infty)}\frac{\dd y}{|y-x|^{1+2s}}\right)\dd x\,.
\]
By elementary integration, since $0 < s < \frac{1}{2}$, we obtain
\[
\int_{\R}\!\int_{\R}\frac{|\Id_I(x)-\Id_I(y)|^{2}}{|x-y|^{1+2s}}\,\dd y\,\dd x = \frac{2\,a^{1-2s}}{s(1-2s)}\,.
\]
This converges to $0$ as $a\to0$. So
$\|\Id_I\|_{H^{s}(\R)}\to0$ as $|I|=a\to0$.
\end{proof}
We also have the following result about the stability of multiplication by indicator functions in $H^{\sigma}(\R)$, see \cite[Lemma 5]{AT01}.
\begin{lemma}
\label{multiplication_indicator_lemma}
Let $\sigma \in [0,\frac{1}{2})$.
Suppose that $T > 0$ and that $f \in H^{\sigma}((0,T)) \cap C([0,T])$. Then
\begin{equation*}
\Id_{[0,T]}f \in H^{\sigma}(\R).
\end{equation*}
\end{lemma}

\subsubsection*{Compactness Results}
We record the following theorems, which give us compact embeddings of various function spaces. First we have the following version of the Rellich--Kondrachov theorem for compact manifolds without boundary, see for example \cite[Proposition 3.4]{Tay11}.
\begin{proposition}[Rellich--Kondrachov for compact manifolds]
Suppose that $M$ is a compact manifold, and let $s \in \R$. Suppose that $\sigma > 0$. The natural inclusion map
\begin{equation*}
j: H^{s + \sigma}(M) \hookrightarrow H^{s}(M)
\end{equation*}
is a compact embedding.
\end{proposition}
We will also use the following version of the Aubin--Lions--Simon lemma, see \cite{Sim87}.

\begin{proposition}[Aubin--Lions--Simon]
Suppose $X,X_0,X_1$ are Banach with $X_0 \subset X \subset X_1$ and suppose that $X_0$ is compactly embedded in $X$ and $X$ is continuously embedded in $X_1$. Suppose $q \in (1,\infty]$. Let
\[
W:= \{u \in L^\infty([0,T];X_0) : \partial_t u \in L^q([0,T];X_1)\}\,.
\]
Then $W$ compactly embeds into $C([0,T];X)$.
\end{proposition}
We also recall that since we work in one dimension, whenever $s > \frac{1}{2}$, we have the embedding $H^s_x \hookrightarrow L^\infty_x$, and $H^s_x$ is a Banach algebra. Moreover, for $s \in [0,\frac{1}{2}]$, we have the estimate
\begin{equation}
\label{Hs_bound}
\|fg\|_{H^s_x} \lesssim \|f\|_{L^\infty_x}\|g\|_{H^s_x} + \|f\|_{H^s_x}\|g\|_{L^\infty_x}\,.
\end{equation}

\section{Construction of Energy Conserving Solutions}
\label{Section:sNLS}

\subsection{Uniform Bounds on the Smoothed NLS}
We construct the solutions as weak limits of the following smoothed NLS
\begin{equation}
\left\{\begin{alignedat}{2}
\ii \partial_t u^\eps + \Delta u^\eps &=  V^{\eps} |u^\eps|^{2}u^\eps, \\
u^\eps(x,0) &= u_0(x) \in H^1(X)\,.
\end{alignedat}\right.
\tag{sNLS}
\end{equation}
where
{ $V^\eps$
\[
V^{\eps}(x)=\sum_{k\in \Z} \frac{1}{\eps} V\left(\frac{x+2\pi k}{\eps}\right)
\]
with $V\in C^\infty(\T)$ such that $\int_\T V \dd x = 1$. The $V^\eps \to \delta$ in $H^{-s}(\T)$ for any $s > \frac{1}2{}$}. We have the following well-posedness result for \eqref{smoothed_NLS_introduction}.
\begin{proposition}
\label{smoothed_well_posedness_proposition}
Suppose $u_0 \in H^1$. Then for each fixed $\eps \in (0,1)$, there is a unique global solution $C(\R;H^1(\T))$ to \eqref{smoothed_NLS_introduction}. Moreover, we have
\begin{align*}
M(u^\eps(t)) &= \int |u^\eps(x,t)|^2 \, \dd x = M(u_0)\,, \\
E^\eps(u^\eps(t)) &= \int |\nabla u^\eps(x,t)|^2 \, \dd x + \frac{1}{2} \int V^{\eps}(x)|u^\eps(x,t)|^{4} \, \dd x = E^\eps(u_0)\,.
\end{align*}
Finally, we have the uniform bounds
\begin{align}
\label{uniformly_H1_bound}
\sup_{\eps \in (0,1)} \|u^\eps\|_{L^\infty_t(\R)H^1_x(\T)} &\leq C(\|u_0\|_{H^1})\,, \\
\label{uniformly_H-1_bound}
\sup_{\eps \in (0,1)} \|\partial_t u^\eps\|_{L^\infty_t(\R) H_x^{-1}(\T)} &\leq C(\|u_0\|_{H^1})\,.
\end{align}
\end{proposition}
\begin{proof}
The existence and uniqueness of local solutions is easily proved using a fixed point argument. Conservation of energy and mass follow from a standard argument, see for example \cite{Tao06}. We make the fixed point argument on the set
\begin{equation*}
\mathcal{B} := \{u \in C([0,T(\eps)];H^1(\T)) : \|u\|_{L^\infty H^1} \leq 2 \|u_0\|_{H^1}\}\,,
\end{equation*}
where $T = T(\eps) \sim \|V^{\eps}\|_{H^1}^{-1} \|u_0\|_{H^1}^{-2}$. We extend the solutions to global solutions using conservation of energy, see for example \cite{Caz03}. So it only remains to prove \eqref{uniformly_H1_bound} and \eqref{uniformly_H-1_bound}. We have
\begin{multline*}
\|u^\eps(t)\|_{H^1_x}^2 = \|\nabla u^\eps(t)\|_{L^2_x}^2 + \|u^\eps(t)\|_{L^2_x}^2 \leq E^\eps(u^\eps(t)) + M(u^\eps(t)) \\
= E^\eps(u^\eps(0)) + M(u^\eps(0))
= \|u_0\|_{H^1}^2 + \frac{1}{2}\int V^\eps(x) |u_0(x)|^{4} \, \dd x\,.
\end{multline*}
Applying Hölder's inequality, using $\|V^{\eps}\|_{L^1} = 1$, and using the Sobolev embedding theorem 
\begin{equation}\label{eq:u-eps-H1-norm}
\|u^\eps(t)\|_{H^1_x}^2 \leq \|u_0\|_{H^1}^2 + \frac{1}{2}\|V^\eps\|_{L^1}\|u_0\|_{L^\infty}^{4} \leq C(\|u_0\|_{H^1})\,.
\end{equation}
Taking a supremum in time, we obtain \eqref{uniformly_H1_bound}. For \eqref{uniformly_H-1_bound}, we use the fact mild solutions are strong solutions to write
\begin{align*}
\|\partial_t u^\eps(t)\|_{H^{-1}_x} \leq \|\Delta u^\eps(t)\|_{H^{-1}_x} + \|V^{\eps} |u^\eps(t)|^{2} u^\eps(t)\|_{H^{-1}_x} \leq \|u^\eps(t)\|_{H^1_x} + \|V^{\eps}\|_{H^{-1}_x} \|u^\eps(t)\|_{H^1_x}^{3}\,.
\end{align*}
Using $V^{\eps} \to \delta \in H^{-1}(\T)$, taking a supremum in time, and applying \eqref{uniformly_H1_bound}, we obtain \eqref{uniformly_H-1_bound}.
\end{proof}

\subsection{Existence of Solutions for the \ref{concentrated_NLS_introduction}}
\label{Section:existence_of_solution}
In this section we prove the existence of solutions to \eqref{concentrated_NLS_introduction}. In what follows, we will denote by $H_w^1$ the space $H^1$ endowed with the weak topology.
\begin{lemma}\label{weak_limits_lemma}
Let $T\in\R$ and  $u_0\in H^{1}(\T)$.  
Let $\{u^\eps\}\subset C([0,T];H^{1}(\T))$ be the sequence of solutions to \eqref{smoothed_NLS_introduction}.  
Then, up to a subsequence, the following hold.
\begin{enumerate}[label=(\roman*),font=\normalfont]
    \item \label{property_i} $u^{\eps_j}\to u$ strongly in $C([0,T];H^{s}(\T))$ for every $s\in(\frac12,1)$.
    \item \label{property_ii} $u^{\eps_j}(t)\rightharpoonup u(t)$ weakly in $H^{1}(\T)$ for every $t\in[0,T]$.
    \item \label{property_iii} $u\in C\bigl([0,T];H^{1}_w(\T)\bigr)$.
\end{enumerate}
\end{lemma}
\begin{proof}
As a first step, we prove the compactness in $C([0,T];H^{s})$.
Set 
\begin{align*}
W&:=\bigl\{v\in L^\infty({[0,T]};H^1(\T)) : \partial_t v\in L^\infty({[0,T];H^{-1}(\T)})\bigr\}\,,\\
\|v\|_{W}&:=\|v\|_{L^\infty_tH^{1}_x}+\|\partial_t v\|_{L^\infty_tH^{-1}_x}\,.
\end{align*}
By Proposition \ref{smoothed_well_posedness_proposition}, the family $\{u^\eps\}$ is bounded in $W$.
For $s\in(\tfrac12,1)$, Rellich--Kondrachov implies that
\[
H^{1}(\T)\xhookrightarrow{\;\text{compact}\;}
H^{s}(\T)\xhookrightarrow{\;\text{continuous}\;}
H^{-1}(\T)\,.
\]

By the Aubin--Lions--Simon compactness lemma, using the compact embedding $H^1(\T) \Subset H^s(\T)$ and the continuous embedding $H^s(\T) \hookrightarrow H^{-1}(\T)$, we obtain that
\[
W \hookrightarrow C([0,T]; H^s(\T))
\]
is compact.
Hence for each fixed $s\in(\tfrac12,1)$, there exists a subsequence $u^{\eps_j^{(s)}}$ converging strongly in $C([0,T];H^{s})$.

Now, for the diagonal extraction, choose $s_n:=1-\frac1n\nearrow1$.  By iterating the first step and taking a standard diagonal subsequence, we obtain a single subsequence (still denoted $u^{\eps_j}$) that converges strongly in $C([0,T];H^{s_n}(\T))$ for every $n\in\N$.  Since $H^{s_{n+1}}(\T)\hookrightarrow H^{s_{n}}(\T)$ continuously, this implies strong convergence in $C([0,T];H^{s}(\T))$ for \emph{all} $s\in(0,1)$, proving \ref{property_i}.

It remains to prove \ref{property_ii} and \ref{property_iii}. Fix $t \in [0,T]$. One has that the sequence $\{u^{\eps_j}(t)\}$ is bounded in $H^1(\T)$. So Banach--Alaoglu implies every subsequence has a further subsequence such that
\begin{equation*}
u^{\eps_{j_k}}(t) \rightharpoonup v
\qquad\text{weakly in }H^1(\T)    
\end{equation*}
for some $v \in H^1(\T)$. By \ref{property_i}, we also have that
\begin{equation*}
u^{\eps_{j_k}}(t) \to u(t)
\qquad\text{strongly in }H^s(\T)
\end{equation*}
for every $s \in (\frac{1}{2},1)$. It follows that $v=u(t)$. Since the weak cluster point is unique, the whole sequence satisfies
\begin{equation*}
u^{\eps_j}(t) \rightharpoonup u(t)
\qquad\text{weakly in }H^1(\T).
\end{equation*}
Property \ref{property_iii} follows for example from \cite[Chapter III, Lemma 1.4]{Tem73}.
\end{proof}

We now show that the function constructed in Lemma \ref{weak_limits_lemma} is a solution of the \eqref{concentrated_NLS_introduction}.
\begin{proposition}
\label{existence_proposition}
Suppose that $u$ is the function constructed in Lemma \ref{weak_limits_lemma}. Then for each time $t\in [0,T]$, $u(t)$ is a mild solution of \eqref{concentrated_NLS_introduction}.
\end{proposition}
\begin{proof}
For notational simplicity, we write $\eps_j = \eps$. Let $t \in [0,T]$ and take $s \in (\frac{1}{2},1)$. We have
\begin{multline}
\label{existence_intermediate}
\left\|u(t) - S(t)u_0 + \ii \int^t_0 S(t-t') \delta|u|^{2} u \, \dd t'\right\|_{H^{-s}_x} \leq \underbrace{\|u(t)-u^\eps(t)\|_{H^{-s}_x}}_{\mathrm{I}} + \underbrace{\|S(t)\left[u_0 - u_0^\eps\right]\|_{H^{-s}_x}}_{\mathrm{II}} \\
+ \underbrace{\left\|u^\eps(t) - S(t)u_0^\eps + \ii \int^t_0 S(t-t') V^\eps|u^\eps|^{2} u^\eps \, \dd t'\right\|_{H^{-s}_x}}_{\mathrm{III}} \\
+ \underbrace{\left\|\ii \int^t_0 S(t-t')\left[V^{\eps} |u^\eps|^{2}u^\eps - \delta|u|^2u \right] \dd t'\right\|_{H^{-s}_x}}_{\mathrm{IV}}.
\end{multline}
By construction, $\mathrm{III}$ is zero because $u^\eps$ is a mild solution of \eqref{smoothed_NLS_introduction}. We also have $\mathrm{II} = 0$. Since the left hand side of \eqref{existence_intermediate} is independent of $\eps$ and $\mathrm{I} \to 0$ by construction, it only remains to show that $\mathrm{IV} \to 0$ as $\eps \to 0$. We have
\begin{equation}
\label{existence_intermediate_2}
\mathrm{IV} \leq \int_0^t \|S(t-t') |u^\eps|^{2}u^\eps (V^{\eps} - \delta)\|_{H^{-s}}\dd t' + \int_0^t \|S(t-t') \left[|u^\eps|^{2}u^\eps - |u|^2u\right] \delta \|_{H^{-s}} \dd t'.
\end{equation}
Recall that the Schrödinger kernel does not change the Sobolev norm of a function. So the first term in \eqref{existence_intermediate_2} is bounded by
\begin{equation*}
t \|u^\eps\|^{3}_{L^{\infty}_{[0,T]}H^1_x}\|V^{\eps} - \delta\|_{H^{-s}_x} \to 0
\end{equation*}
as $\eps \to 0$, where we have used \eqref{eq:u-eps-H1-norm}. For the second term of \eqref{existence_intermediate_2}, we have
\begin{equation}
\label{existence_intermediate_3}
\| \left[|u^\eps|^{2}u^\eps(t') - |u|^{2}u(t')\right]\delta \|_{H^{-s}_x} \lesssim \|\delta\|_{H^{-s}_x}\|u^\eps(t') - u(t')\|_{H^{s}_x} (\|u^\eps(t')\|_{H^{1}_x} + \|u(t')\|_{H^{1}_x})^{2}.
\end{equation}

So we have that the second term in \eqref{existence_intermediate_2} is less than or equal to
\begin{equation*}
T \|\delta\|_{H^{-s}} (\|u^{\eps}\|_{L^\infty_{[0,T]} H^1(\T)} + \|u\|_{L^\infty_{[0,T]} H^1(\T)})^2  \|u^{\eps} - u\|_{L^\infty_{[0,T]}H^{s}(\T)} \to 0\,.
\end{equation*}
Here we have used \eqref{uniformly_H1_bound} and that $u^\eps \to u$ in $C([0,T];H^s(\T))$. So $\mathrm{IV} \to 0$ as $\eps \to 0$, and $u$ is a mild solution of \eqref{concentrated_NLS_introduction} for each time $t \in [0,T]$.
\end{proof}

\subsection{Conservation Laws for a Solution of the \ref{concentrated_NLS_introduction}}
Recall that the mass and the energy for the \eqref{concentrated_NLS_introduction} are given by
\begin{align*}
M(u) &:= \int_{X} |u(x,t)|^2 \, \dd x \,,\\
E(u) &:= \int_{X} |\nabla u(x,t)|^2 \, \dd x + \frac{1}{2} |u(0,t)|^4. 
\end{align*}
In this section, we show that solutions to the \eqref{concentrated_NLS_introduction} constructed in Section \ref{Section:existence_of_solution} conserve both the mass and the energy.
\begin{proposition}
\label{conservation_of_mass_proposition}
The solution $u$ constructed in Lemma \ref{weak_limits_lemma} conserves the mass. In other words, 
\begin{equation*}
M(u(t)) = M(u_0)
\end{equation*}
for any time $t \in [0,T]$, for any $T > 0$.
\end{proposition}
\begin{proof}
By construction, we know from Lemma \ref{weak_limits_lemma}, the trivial bound $\|\cdot\|_{L^2} \leq \|\cdot\|_{H^{s}}$, and continuity in time that, for any $t \in [0,T]$,
\begin{equation*}
\lim_{j \to \infty} \|u^{\eps_j}(t)\|_{L^2}^2 = \|u(t)\|_{L^2}^2\,.
\end{equation*}
However, $u^{\eps_j}$ is mass preserving, so we know
\begin{equation*}
\|u^{\eps_j}(t)\|_{L^2}^2 = \|u_0\|_{L^2}^2\,,
\end{equation*}
so the result follows.
\end{proof}
To prove the conservation of energy, we define the corresponding potential energies
\begin{align*}
E_p(u(t)) &:= \frac{1}{2}|u(0,t)|^4, \\
E_p^\eps(u(t)) &:= \frac{1}{2} \int V^{\eps}(x)|u(x,t)|^4 \, \dd x \,.
\end{align*}
Here we note that $E_p(u(t))$ is well-defined because $u(t) \in H^s$ for $s \in (\frac{1}{2},1]$.

\begin{lemma}
\label{potential_energy_convergence_lemma}
$E_p^{\eps_j}(u^{\eps_j}(t)) \to E_p(u(t))$ for all $t \in [0,T]$, for any $T > 0$.
\end{lemma}
\begin{proof}
For simplicity of notation, we write $\eps_j =\eps$. Recall that $\int V_\eps \, \dd x = 1$ and $V_\eps \geq 0$. Note that
\begin{equation*}
\int V_\eps(x) \left[ |u^\eps(x,t)|^4 - |u(0,t)|^4 \right] \, \dd x \leq \left( \|u^\eps(t)\|_{L^\infty_x} + \|u(t)\|_{L^\infty_x}\right)^3 \left[\int V_\eps(x)\left|u^\eps(x,t) - u(0,t)\right| \, \dd x\right].
\end{equation*}
So it suffices to show
\begin{equation*}
\left| \int V_\eps(x)\left|u^\eps(x,t) - u(0,t)\right| \, \dd x \right|
\leq \underbrace{\int V_{\eps}(x)\left|u^{\eps}(x,t) - u(x,t)\right| \dd x}_{\leq \|u^\eps(t) - u(t)\|_{L^\infty_x}\|V_{\eps}\|_{L^1} \to 0} + \underbrace{\int V_{\eps}(x)\left|u(x,t) - u(0,t)\right| \dd x}_{\to 0}.
\end{equation*}
For the second term, we use that $x \mapsto u(x,t)$ is continuous.
\end{proof}

\begin{lemma}\label{energy_inequality_lemma}
	Let $u$ be the function constructed in Lemma
	\ref{weak_limits_lemma}. Then
	\[
		E(u(t))\leq E(u_0)
	\]
	for every $t\in[0,T]$.
\end{lemma}
\begin{proof}
Again, for simplicity of notation, we write $\eps_j = \eps$. We have
\begin{multline*}
M(u(t)) + E(u(t)) - E_p(u(t)) = \|u(t)\|^2_{H^1} \leq \liminf_{\eps \to 0} \|u^\eps(t)\|^2_{H^1} \\
= \liminf_{\eps \to 0} \left[ M(u^\eps (t)) + E(u^\eps(t)) - E_p(u^\eps(t)) \right] = M(u_0) + E(u_0) - E_p(u(t))\,.
\end{multline*}
The first inequality is a consequence of the weak convergence of $u^{\eps_j}(t)$ to $u(t)$ in $H^1(\T)$. The final line follows from recalling that the initial condition of $u^\eps$ is $u_0$, that energy and mass are conserved for the approximate equation, and Lemma \ref{potential_energy_convergence_lemma}. The result follows from recalling that the mass of $u$ is also conserved.
\end{proof}

\begin{proposition}
\label{conservation_of_energy_proposition}
The solution $u$ constructed in Lemma \ref{weak_limits_lemma} conserves the energy. In other words, 
\begin{equation*}
E(u(t)) = E(u_0)
\end{equation*}
for any time $t \in [0,T]$, for any $T > 0$.
\end{proposition}
\begin{proof}
Note that all of our proofs also show the existence of a solution for negative time as well, and that the decay of energy holds for all $|t| < T_*$, where $T_*$ is the time of existence. Suppose we had a time $T$ such that $E(u(T)) < E(u(0))$. Following \cite[Proposition 3.5]{HGKV24}, we can define $v(x,t) := u(x,T+t)$. Then we obtain a solution with initial condition $v(0)$ such that $E(v(-T)) > E(v(0))$, which is a contradiction. So energy must be conserved.
\end{proof}
In the following proposition, we upgrade weak continuity to strong continuity.
\begin{proposition}
\label{continuity_H1_proposition}
Let $u$ be the function constructed in Lemma \ref{weak_limits_lemma}. Then $u \in C([0,T];H^1(\T))$.
\end{proposition}
\begin{proof}
The proof is similar to part of the proof \cite[Proposition 3.5]{HGKV24}. Recall that we already have $u \in C([0,T];H^s(\T)) \cap C([0,T];H^1_w(\T))$. Since $u \in C([0,T];H^s(\T))$ for $s \in (\frac{1}{2},1)$, the Sobolev embedding theorem implies that the mapping
\begin{equation*}
t \mapsto E_p(u(t))
\end{equation*}
is continuous. 
So
\begin{equation*}
t \mapsto \|u(t)\|^2_{H^1} = M(u_0) + E(u_0) - E_{p}(u(t))
\end{equation*}
is continuous, where we have used conservation of mass and energy. This and the Radon--Riesz theorem allow us to upgrade $C([0,T];H^1_w(\T))$ to $C([0,T];H^1(\T))$.
\end{proof}

Putting together all the propositions in Section \ref{Section:sNLS}, we obtain the following existence result for the \eqref{concentrated_NLS_introduction}.
\begin{theorem}[Existence of energy conserving solutions to cNLS]
\label{existence_theorem}
Let $u_0 \in H^1(\T)$. Then for any $T > 0$, there is a function $u \in C([0,T];H^1(\T))$ such that $u$ solves \eqref{concentrated_NLS_introduction}. Moreover, the function $u$ conserves mass and energy.
\end{theorem}

\begin{remark}
The convergence of solutions to the smoothed NLS to solutions to \eqref{concentrated_NLS_introduction} is similar to the main result of \cite{CFNT14}, which was partially extended to three dimensions in \cite{CFNT17}. It was also central to the first author's derivation of the \eqref{concentrated_NLS_introduction} with Adami in \cite{AL25}. We also extend our convergence to strong convergence of the entire sequence in $C([0,T];H^s(\T))$ for any $s \in (\frac{1}{2},1)$. See Corollary \ref{strong_convergence_corollary} for a precise statement. 
\end{remark}

\subsection{Uniqueness of solutions}
\label{Section:Uniqueness}
To prove uniqueness of global solutions, we exploit the concentrated nature of the $\delta$ function. In particular, since the $\delta$ function is concentrated at zero, any solution in $L^\infty([0,T];H^1(\T))$ is completely determined by its value at $x = 0$. Therefore, to prove uniqueness, it suffices to show that the Volterra integral equation given by
\begin{equation}
\label{NLS_Volterra}
u(0,t) = S(t)u_0(0) - \ii \int_0^t \sum_{n \in \Z} \e^{-\ii n^2 (t-t')} |u(0,t')|^2u(0,t') \, \dd t',
\end{equation}
has a unique global solution. We will alternatively call \eqref{NLS_Volterra} the {\it Volterra equation} or {\it charge equation}. To simplify the notation, in what follows we will write $q(t) := u(0,t)$. Recall the notation
\begin{equation}
\label{S_delta_definition}
S^{\delta}(t) := \sum_{n \in \Z} \e^{-\ii n^2 t}.
\end{equation}
A direct computation using the Fourier coefficients of the kernel yields the following result.
\begin{lemma}
\label{Schrodinger_kernel_time_regularity}
Suppose $s > \frac{1}{4}$. Then $S^{\delta} \in H^{-s}(\T)$.
\end{lemma}
\begin{remark}
Since the Schrödinger kernel is periodic in time, we can only ever consider its regularity in $\R$ when multiplying by with a cut-off function. This is important, since when analysing the equation \eqref{NLS_Volterra}, we will need to consider convolutions on the real line.
\end{remark}
We recall the following basic lemma about extensions of periodic distributions.
\begin{lemma}
Suppose that $D \in H^{s}(\T)$ for $s \in \R$ and let $\chi$ be a smooth cut-off function. Then $\|D\chi\|_{H^{s}(\R)} \lesssim_{\chi} \|D\|_{H^{s}(\T)}$.
\end{lemma}
We now collect some results we need to analyse the charge equation \eqref{NLS_Volterra}. Before proceeding, let us fix a smooth cut-off function $\chi : \R \to [0,1]$ satisfying
\begin{equation}
\label{cut-off_definition}
\chi(x) = 
\begin{cases}
    1 \quad \text{if } x \in [-1,1], \\
    0 \quad \text{if } x \notin [-2,2]
\end{cases}
\end{equation}
We also define the function $\chi_T := \chi(\frac{x}{T})$. In what follows, we will suppress the $\chi$ dependence of constants. We have the following lemma about convolution with $S^\delta$.

\begin{lemma}
\label{convolution_lemma}
Suppose that $T \in (0,1]$. Then for any $\sigma \in [0,\frac{1}{2})$, one has
\begin{equation*}
\|(\chi_T S^{\delta}) * (\Id_{[0,T]} F)\|_{H^\sigma_t(\R)} \lesssim T^\frac{1}{2} \|\Id_{[0,T]} F\|_{H^\sigma_t(\R)}.
\end{equation*}
\end{lemma}
\begin{proof}

We have
\begin{equation*}
\|(\chi_T S^{\delta}) * (\Id_{[0,T]} F)\|^2_{H^s} = \int_{\R} \langle \omega \rangle^{2s} |\widehat{\chi_T S^\delta}|^2(\omega) |\widehat{\Id_{[0,T]}F}|^2(\omega) \, \dd \omega.
\end{equation*}
So it suffices to show that 
\begin{equation*}
\|\widehat{\chi_T S^\delta}\|_{L^\infty} \lesssim T^{\frac{1}{2}}.
\end{equation*}
First, recall that $\hat{\chi}_T(\omega) = T \hat{\chi}(T\omega)$. It follows that
\begin{equation*}
\widehat{\chi_T S^\delta}(\omega) = T \sum_n \hat{\chi}(T(\omega + n^2)).
\end{equation*}
Recall that for any $c, R \in \R$, one has the bound
\begin{equation}
\label{trivial_bound}
\# \{n \in \Z : |c + n^2| \leq R\} \lesssim 1+ \sqrt{R}.
\end{equation}
Since $\chi \in C^\infty_0(\R)$, one has that its Fourier transform is Schwartz, and therefore
\begin{equation}
\label{Schwartz_decay}
|\hat{\chi}(T(\omega + n^2))| \lesssim \frac{1}{(1 + T|\omega + n^2|)^2}.
\end{equation}
Let us decompose the sum into dyadic blocks with $\{n : 2^{j-1} \leq 1 + T|\omega + n^2| < 2^{j}\}$ for $j\in\N$. By \eqref{trivial_bound}, the number of such $n \in \Z$ is bounded, up to constants, by $1+\sqrt\frac{2^{j}}{T}$. Using \eqref{Schwartz_decay} and $T \in (0,1]$, one has that
\begin{equation*}
\|\widehat{\chi_T S^\delta}\|_{L^\infty} \lesssim T\sum_{j \geq 1} 2^{-2(j-1)} \left( 1+\frac{2^{\frac{j}{2}}}{T^{\frac{1}{2}}} \right)\lesssim T+T^{\frac{1}{2}} \lesssim T^{\frac{1}{2}}.
\end{equation*}
\end{proof}
We are thus able to prove the following uniqueness result.
\begin{lemma}
\label{uniqueness_bounded_charge_lemma}
Let $T > 0$ and suppose that there is a solution to \eqref{NLS_Volterra} for $u_0 \in H^1(\T)$ which is in $L^\infty([0,T])$. Then the solution to the charge equation is unique.
\end{lemma}
\begin{proof}
Suppose that we have two solutions $q_1,q_2 \in L^\infty([0,T])$ to the charge equation \eqref{NLS_Volterra}, and fix $\tau$ sufficiently small to be determined later. Then
\begin{multline*}
\|q_1 - q_2\|_{L^2_{[0,\tau]}} = \|(\chi_\tau S^{\delta}) * (\Id_{[0,\tau]} (|q_1|^2q_1 - |q_2|^2q_2))\|_{L^2([0,\tau])} \lesssim \tau^{ \frac{1}{2}} \||q_1|^2q_1 - |q_2|^2q_2\|_{L^2_{[0,\tau]}} \\
\leq \tau^{\frac{1}{2}}\left(\|q_1\|_{L^\infty_{[0,\tau]}} + \|q_2\|_{L^\infty_{[0,\tau]}}\right)^2  \|q_1 - q_2\|_{L^2_{[0,\tau]}}.
\end{multline*}
Taking $\tau$ to be sufficiently small, one has that $q_1 = q_2$ almost everywhere. Since the $L^\infty([0,T])$ norm of the solution to \eqref{NLS_Volterra} is bounded, we can iterate the argument to the maximum time of existence.
\end{proof}
We can now prove Theorem \ref{existence_uniqueness_theorem}
\begin{proof}[Proof of Theorem \ref{existence_uniqueness_theorem}]
Since any solution to the \eqref{concentrated_NLS_introduction} is uniquely determined by the solution \eqref{NLS_Volterra}, the result follows from Lemma \ref{uniqueness_bounded_charge_lemma} and Theorem \ref{existence_theorem}. In particular, we use the fact that the mapping $t \mapsto u(0,t)$ is continuous for the solution constructed in Theorem \ref{existence_theorem}.

It only remains to prove that the solution depends continuously on the initial data. To prove this, fix $T > 0$ and take $u_{0,n} \to u_0$ in $H^1(\T)$. We know that for each $u_{0,n}$, there is a unique solution $u_{n} \in C([0,T];H^1(\T))$ which conserves the mass and the energy. Moreover
\begin{equation*}
\|u_n\|_{L^\infty([0,T])H^1(\T)} \leq C(\|u_{0,n}\|_{H^1}) \leq C(\|u_0\|_{H^1}).
\end{equation*}
Here we have used that $u_{0,n} \to u_0$ in $H^1(\T)$. By arguing as in the rest of this section, we thus get that there is a subsequence such that
\begin{equation*}
u_{n_j} \to u \quad \text{in} \,\, C([0,T];H^s) \,\, \text{for all} \,\,s < 1.
\end{equation*}
Moreover, recall that a sequence converges if every subsequence has a further subsequence that converges to the same element. Compactness and uniqueness guarantees this, so the entire sequence converges in $C([0,T];H^s(\T))$. Moreover, since $s > \frac{1}{2}$, this implies that
\begin{equation*}
\sup_{t \in [0,T]}|u_{n}(0,t) - u(0,t)| \to 0.
\end{equation*}
By conservation of energy, it thus follows that 
\begin{equation}
\label{norm_convergence}
\sup_{t \in [0,T]} |\|u_n(t)\|_{H^1} - \|u(t)\|_{H^1}| \to 0.
\end{equation}
Suppose for a contradiction that
\begin{equation*}
\|u_n - u\|_{L^\infty([0,T]) H^1} \not\to 0.
\end{equation*}
By definition, there exists a sequence of times $t_n$ such that
\begin{equation}
\label{convergence_contradiction}
\|u_n(t_n) - u(t_n)\|_{H^1} > \eta
\end{equation}
for some $\eta > 0$. Since $[0,T]$ is compact, we can assume that up to a subsequence, also denoted $t_n$, $t_n \to t_*$. Since
\begin{equation*}
\|u_n(t_n) - u(t_n)\|_{H^1} \leq \|u_n(t_n) - u(t_*)\|_{H^1} + \|u(t_*) - u(t_n)\|_{H^1},
\end{equation*}
and $u \in C([0,T];H^1)$, to show that \eqref{convergence_contradiction} does not hold, it suffices to show that
\begin{equation}
\label{contradiction_proof}
\|u_n(t_n) - u(t_*)\|_{H^1} \to 0.
\end{equation}
Recall $(u_n(t_n))$ is a uniformly bounded sequence in $H^1(\T)$, which follows from conservation of mass and energy, and that $u_{0,n} \to u$ in $H^1(\T)$. Therefore, there is a subsequence, also indexed as $n$, such that $u_n(t_n) \rightharpoonup w$ in $H^1(\T)$. One has that
\begin{equation*}
\|u_n(t_n) - u(t_n)\|_{H^s} \leq \|u_n-u\|_{L^\infty([0,T]) H^s} \to 0
\end{equation*}
 for any $s < 1$. Since
\begin{equation*}
\|u_n(t_n) - u(t_*)\|_{H^s} \leq \|u_n(t_n) - u(t_n)\|_{H^s} + \|u(t_n) - u(t_*)\|_{H^s} \to 0.
\end{equation*}
Here we use the continuity of $u$ in $H^s(\T)$. Therefore $u^n(t_n) \rightharpoonup u(t)$ in $H^s(\T)$ for any $s < 1$, so uniqueness of weak limits ensures that $w = u(t_*)$. Moreover, one has that
\begin{equation*}
\left| \|u_n(t_n)\|_{H^1} - \|u(t_*)\|_{H^1}\right| \leq \left| \|u_n(t_n)\|_{H^1} - \|u(t_n)\|_{H^1}\right| + \left| \|u(t_n)\|_{H^1} - \|u(t_*)\|_{H^1}\right|
\end{equation*}
The first term on the right hand side converges to zero by \eqref{norm_convergence}, and the second because of continuity in $H^1$ norm. Therefore we have $u_n(t_n) \to u(t_*)$ by Radon--Riesz, completing the proof of \eqref{contradiction_proof}.
\end{proof}

We also have the following corollary about the strong concentrated limit of the \eqref{smoothed_NLS_introduction}, which follows from a similar argument.
\begin{corollary}
\label{strong_convergence_corollary}
Suppose that $u_0 \in H^1(\T)$. Then for any $T > 0$, we have
\begin{equation*}
\lim_{\eps \to 0} \|u^{\eps} - u\|_{L^\infty_{[0,T]} H^1(\T)} = 0 \,.
\end{equation*}
\end{corollary}

\section{Well-posedness theory below the energy space} \label{Section:LWP}

\subsection{Existence of solutions}
As in the global case, we again analyse the problem of the Volterra equation
\begin{equation}
\label{NLS_Volterra_local}
u(0,t) = S(t)u_0(0) - \ii \int_0^t \sum_{n \in \Z} \e^{-\ii n^2 (t-t')} |u(0,t')|^2u(0,t') \, \dd t'.
\end{equation}
Again, we will write $q(t) := u(0,t)$. Before proceeding with the analysis of the charge equation, we need to prove the following result about the kernel $S^\delta$ defined in \eqref{S_delta_definition}.

\begin{lemma}
\label{H-s_to_zero_lemma}
Suppose that $\sigma > \frac{1}{4}$. Then for $T \in (0,1]$ and for $\chi$ as in \eqref{cut-off_definition}, we have
\begin{equation*}
\|\chi_T S^\delta\|_{H^{-\sigma}(\R)} \lesssim_{\chi} T^{\frac{1}{4}}.
\end{equation*}
\end{lemma}
\begin{proof}
By direct computation, one has
\begin{equation*}
\|\chi_T S^\delta\|_{H^{-\sigma}(\R)}^2 \sim \int_{\R} \langle \xi \rangle^{-2\sigma} \left| \sum_{n} T\hat{\chi}(T(\xi - n^2))\right|^2 \, \dd \xi.
\end{equation*}
By the proof of Lemma \ref{convolution_lemma}, we have that 
\begin{equation*}
 \left| \sum_{n} T\hat{\chi}(T(\xi - n^2))\right| \lesssim T^{\frac{1}{2}}.
\end{equation*}
So it suffices to show we have
\begin{equation}
\label{uniform_time_bound}
\int_{\R} \langle \xi \rangle^{-2\sigma} \left| \sum_{n} T\hat{\chi}(T(\xi - n^2))\right| \, \dd \xi < C < \infty,
\end{equation}
uniformly in time. To show \eqref{uniform_time_bound}, we note that it suffices to prove
\begin{equation}
\label{sum_bounded_uniformly}
\sum_n \langle n^2 - a\rangle^{-2\sigma} < C < \infty
\end{equation}
uniformly in $a \in \R$. Indeed, once we have \eqref{sum_bounded_uniformly}, one can apply the triangle inequality and change variables $v = T(\xi - n^2)$ in the integrand of \eqref{uniform_time_bound} to get
\begin{align*}
\int_{\R}
\langle\xi\rangle^{-2\sigma}
\left|
\sum_n T\widehat\chi(T(\xi-n^2))
\right|
\,\dd\xi
&\le
\sum_n
\int_{\R}
\langle\xi\rangle^{-2\sigma}
T
\left|
\widehat\chi(T(\xi-n^2))
\right|
\,\dd\xi\\
&=
\int_{\R}
|\widehat\chi(v)|
\sum_n
\left\langle
n^2+\frac{v}{T}
\right\rangle^{-2\sigma}
\,\dd v
\lesssim
\|\widehat\chi\|_{L^1},
\end{align*}
where we have used Tonelli's theorem and the fact that $\langle x \rangle = \langle -x \rangle$. So it only remains to show \eqref{sum_bounded_uniformly}. For this, we recall \eqref{trivial_bound} and write
\begin{equation*}
    \sum_n \langle n^2 - a\rangle^{-2\sigma} \lesssim \sum_{j} 2^{-2\sigma j} \#\{n : |n^2-a| \leq 2^{j}\} \lesssim \sum_j 2^{j(\frac{1}{2} - 2\sigma)} + 2^{-2\sigma j} < \infty
\end{equation*}
when $\sigma > \frac{1}{4}$. This completes the proof of \eqref{sum_bounded_uniformly}.
\end{proof}

We also have the following result, which gives the time regularity of the free flow computed at $0$.
\begin{lemma}
\label{regularity_free_evolution_lemma}
Suppose that $u_0 \in H^{s}(\T)$ for $s > \frac{1}{2}$. Then the mapping $ t \mapsto S(t)u_0(0)$ is in $H^{\sigma}_{t,\,loc}(\R)$ for all $\sigma\leq\frac{s}{2}$.
\end{lemma}
\begin{proof}
One has that
\begin{equation*}
S(t)u_0(0) = \sum_{\omega \in \Z} \e^{-\ii {\omega}^2 t} \hat{u}_0({\omega}).
\end{equation*}
Therefore
\begin{equation*}
\|S(t)u_0(0)\|_{H^\sigma_t(\T)}^2 = \sum_{\omega = \pm n^2} \langle {\omega} \rangle^{2\sigma} |\widehat{S(t)u_0(0)}({\omega})|^2 \lesssim \sum_{n \in \Z} \langle n \rangle^{4\sigma} |\hat{u}_0(n)|^2.
\end{equation*}
This finite if and only if $\sigma \leq \frac{s}{2}$.
\end{proof}

We have the following result, which guarantees the local existence of the charge.
\begin{lemma}
\label{free_regularity_lemma}
Suppose $u_0 \in H^{s}(\T)$ with $s \in (1/2,1)$. Let $\sigma \in (\frac{1}{4},\frac{1}{2}s]$. Then there is some $T = T(\|u_0\|_{H^s(\T)}) > 0$ such that there is a unique charge $q \in H^{\sigma}([0,T]) \cap L^\infty([0,T])$ that solves \eqref{NLS_Volterra_local} on $[0,T]$. Moreover, the charge depends continuously on the initial condition.
\end{lemma}
\begin{proof}
We make the contraction argument on the set 
\begin{equation*}
    B:=\{q' \in L^\infty_t \cap H^\sigma_t: \max\{\|q'\|_{L^\infty_t([0,\tau])}, \|q'\|_{H^\sigma_t([0,\tau])}\} \leq 2\|u_0\|_{H^s_x(\T)}\},
\end{equation*}
for $\tau$ sufficiently small to be determined later. Define the map
\begin{equation*}
    \Gamma q := S(t)u_0(0) - \ii \int_0^t S^\delta(t-t') |q|^2q(t') \, \dd t'.
\end{equation*}
We note that one has that $\max\{\|S(t)u_0(0)\|_{L^\infty([0,\tau])}, \|S(t)u_0(0)\|_{H^\sigma([0,\tau])}\}$ is bounded by $\|u_0\|_{H^s(\T)}$ by the local well posed theory of the linear Schr\"odinger equation and Lemma \ref{regularity_free_evolution_lemma}. So to prove that the map is well-defined, we focus on the second term in $\Gamma$. By duality
\begin{equation}
\label{Linfty_charge_bound}
\left\|\int_0^t S^\delta(t-t') |q|^2q(t') \, \dd t'\right\|_{L^\infty_t} \leq \|\chi_\tau S^\delta\|_{H^{-\sigma}_t} \|\Id_{[0,T]}|q|^2q\|_{H^{\sigma}_t} \leq C(\chi) \tau^{\frac{1}{4}} \|\Id_{[0,T]}q\|_{H^\sigma_t}\|\Id_{[0,T]}q\|_{L^\infty_t}^2.
\end{equation}
By Lemma \ref{convolution_lemma}, we have
\begin{equation*}
\left\|\int_0^t S^\delta(t-t') |q|^2q(t')\right\|_{H^\sigma_t} \lesssim \tau^{\frac{1}{2}}\|\Id_{[0,\tau]}q\|_{H^\sigma_t} \|\Id_{[0,\tau]}q\|^2_{L^\infty_t}.
\end{equation*}
Picking $\tau$ sufficiently small, $\Gamma$ is well-defined. Similarly, one gets a contraction as well. The continuity follows from the fact that we constructed the charge via a contraction mapping argument.
\end{proof}
\begin{remark}
\label{blow-up_remark}
By standard Volterra integral equation theory, it follows that the solution satisfies a blow-up criterion on the size of $|q(t)|$; see for example \cite{Mil71}.
\end{remark}
\begin{remark}
\label{continuity_remark}
We are able to upgrade the charge from being $L^\infty([0,T])$ to being $C_t([0,T])$. Indeed, for $s < t$, we write
\begin{multline}
\label{LWP_continuity_1}
\int_0^t S^\delta(t-t') |q|^2q(t') \, \dd t' - \int_0^{s} S^\delta(s-t') |q|^2q(t') \, \dd t'  \\ 
= \int_{s}^t S^\delta(t-t') |q|^2q(t')  \, \dd t' + 
\int_0^{s} \left[S^\delta(t-t')-S^\delta(s-t')\right] |q|^2q(t') \, \dd t'.
\end{multline}
In $L^\infty$ norm, the first term goes to zero as $s \to t$ by arguing as in \eqref{Linfty_charge_bound}. The second term in \eqref{LWP_continuity_1} goes to zero in $L^\infty$ as $s \to t$ by the continuity of Sobolev norms under translation. Here we use that $|q|^2q \Id_{[0,T]} \in H^{\sigma}(\R)$ and the fact $S^\delta \in H^{-\sigma}_{\mathrm{loc}}(\R)$.
\end{remark}

By arguing similarly to the proof of Proposition \ref{existence_proposition}, one obtains the following result.
\begin{proposition}
\label{local_existence_proposition}
Suppose that $u_0 \in H^s(\T)$ for $s \in (\frac{1}{2},1)$. Then for some $T = T(\|u_0\|_{H^s(\T)})$ and for any $\eps > 0$, there is a function $u \in L^\infty([0,T];H^{-\frac{1}{2}-\eps}(\T))$ such that $u$ solves the \eqref{concentrated_NLS_introduction}.
\end{proposition}

\subsection{Uniqueness of local solutions}
To prove uniqueness of local solutions of the \eqref{concentrated_NLS_introduction}, we need to show that we have persistence of regularity, which ensures that the solution is completely determined by its initial condition. Before proceeding, we require the following technical result, which is proved in Appendix \ref{SEC:proposition}.
\begin{proposition}
\label{sobolev_compact_proposition}
Suppose that $f \in H^s(\R)$ has compact support. Then
\begin{equation*}
\sum_{\omega \in \Z} \langle \omega \rangle^{2s} |\hat{f}(\omega)|^2 \lesssim \|f\|_{H^s(\R)}^2 < \infty.
\end{equation*}
Here the implied constant depends on the measure of the support of $f$.
\end{proposition}
\begin{remark}
When applying Proposition \ref{sobolev_compact_proposition}, we will apply to functions of the form $F = \Id_{[0,T]}|q|^2q$. This will allow us to extend an argument of Adami and Teta from \cite{AT01}, which originally held for functions defined on the line.
\end{remark}
\begin{proposition}[Persistence of regularity for local solutions]
\label{persistence_regularity_local_proposition}
Suppose that $u$ is the solution constructed in Proposition \ref{local_existence_proposition}. Then $u \in L^\infty([0,T];H^s(\T))$.
\end{proposition}
\begin{proof}
We argue similarly to \cite[(2.18)]{AT01}. One has persistence of regularity and continuity in time for the free portion of the solution. One computes that the nonlinear part $u^{\mathrm{NL}}$ of the solution satisfies
\begin{equation*}
\|u^{\mathrm{NL}}(t)\|^2_{{H}^\beta(\T)} \sim \sum_{k} (1 + |k|^{2\beta}) \left|\int_0^t \dd t' \e^{-\ii k^2 (t-t')} |q|^2q(t')\right|^2.
\end{equation*}
For the first term in the sum, one bounds by 
\begin{equation*}
\sum_k \left|\int_0^t \dd t' \, \e^{\ii k^2 t'} |q|^2q(t')\right|^2 \leq \sum_k \left|\int_0^t \dd t' \e^{\ii k t'} |q|^2q(t')\right|^2 \lesssim \|\widehat{|q|^2q}\|_{L^2} < \infty.
\end{equation*}
Here we used Proposition \ref{sobolev_compact_proposition}.\\ 
For the second term, we make the change of summation variable $k^2 \mapsto k$ and write
\begin{equation*}
\|u\|^2_{\dot{H}^\beta(\T)} = \sum_{k} {|k|}^\beta \left|\widehat {|q|^2q}(-k)\right|^2 \lesssim \|\Id_{[0,T]}q\|^2_{H^{\frac{\beta}{2}}}.
\end{equation*}
$\langle \omega \rangle$ is invariant under taking $\omega \to -\omega$, so fine. Here we used Proposition \ref{sobolev_compact_proposition}. It follows that $u \in H^\beta(\T)$ if $\frac{\beta}{2} \leq \frac{s}{2}$. In particular, we have that $u \in L^\infty_{[0,T]} H^s(\T)$
\end{proof}
\begin{remark}
    By arguing similarly to \cite[(2.19)]{AT01} (but adapted to our situation), one also obtains that Sobolev norm is continuous in time. This is a consequence of the charge also being continuous in time.
\end{remark}
\begin{remark}
We note that a similar argument to the proof of Proposition \ref{persistence_regularity_local_proposition} can be used to show that we have Hadamard stability. Indeed, one has that $\widehat{u^{\mathrm{NL}} - v^{\mathrm{NL}}} = \widehat{u^{\mathrm{NL}}} - \widehat{v^{\mathrm{NL}}}$. Then arguing as in Proposition \ref{persistence_regularity_local_proposition}, one has
\begin{equation*}
\|u^{\mathrm{NL}}(t) - v^{\mathrm{NL}}(t)\|_{H^s} \leq \|\Id_{[0,t]}(|q_u|^2q_u - |q_v|^2q_v)\|_{H^{\frac{s}{2}}(\R)}\lesssim_{u_0,v_0} \|q_u - q_v\|_{H^{\frac{s}{2}}([0,T])} + \|q_u-q_v\|_{L^\infty([0,T])}.
\end{equation*}
Recalling that the solution depends continuously on the initial condition, the result follows from taking a supremum in time. Here we used that the linear solution is Hadamard stable.
\end{remark}
We thus obtain the following corollary, since the solution is now completely determined by its value at zero.
\begin{corollary}
\label{uniqueness_local_solution_corollary}
Suppose that $u_0 \in H^s(\T)$ for $s \in (\frac{1}{2},1)$. Then for some $T = T(\|u_0\|_{H^s(\T)})$ and for any $\eps > 0$, there is a unique function $u \in L^\infty([0,T];H^{s}(\T))$ such that $u$ solves the \eqref{concentrated_NLS_introduction}.
\end{corollary}

\subsection{Conservation of mass}
To complete the proof of Theorem \ref{LWP_theorem}, it remains only to prove that we have conservation of mass. We note that because we have only proved persistence of regularity in $H^1(\T)$, and not for higher Sobolev spaces, it is not clear that we can use a standard density approach where one uses the equivalence of strong and weak solutions. Instead, we use a truncated initial data approach. We thus prove the following result.
\begin{proposition}[Conservation of mass]
\label{conservation_mass_proposition}
Suppose that $u_0 \in H^s(\T)$ for $s \in (\frac{1}{2},1)$. For the local time of existence $T > 0$, let $u \in L^\infty([0,T]; H^s(\T))$ be the solution to \eqref{concentrated_NLS_introduction}. Then we have $\|u(t)\|_{L^2_x(\T)} = \|u_0\|_{L^2_x(\T)}$ for any $ t \in [0,T]$.
\end{proposition}
\begin{proof}
We consider the concentrated NLS given by
\begin{equation}
\label{truncated_CNLS_equation}
\begin{cases}
\ii \partial_t u^N = -\Delta u^N + \delta |u^N|^2u^N, \\
u^N_0 := \displaystyle\sum_{|n| \leq N} \hat{u}_0(n) \e^{\ii n x}\in H^\infty(\T).
\end{cases}
\end{equation}
We note that $u^N_0 \to u_0$ in $H^s(\T)$ and since $u^N_0 \in H^1(\T)$, one has a unique (global) solution to \eqref{truncated_CNLS_equation} that lives in $L^\infty_t([0,T]) H^1_x(\T)$. Moreover, this solution conserves mass. We note that the charge $q^N(t) := u^N(0,t)$ is bounded in $L^\infty([0,T]) \cap H^\frac{s}{2}([0,T])$ uniformly in $N$, which follows from arguing as in the proof of Lemma \ref{free_regularity_lemma}. It suffices to prove that
\begin{align*}
\left|\|u(0)\|_{L^2(\T)} - \|u^N(0)\|_{L^2(\T)}\right| \to 0, \\
\left|\|u^N(0)\|_{L^2(\T)} - \|u^N(t)\|_{L^2(\T)}\right| \to 0, \\
\left|\|u^N(t)\|_{L^2(\T)} - \|u(t)\|_{L^2(\T)}\right| \to 0.
\end{align*}
The first term follows from the definition of $u^N_0$. Using the conservation of mass of solutions with initial condition in $H^1(\T)$, the second term equals $0$. Using the reverse triangle inequality in the third term, it suffices to consider the norm of the difference. By arguing as in the proof of Proposition \ref{persistence_regularity_local_proposition}, one notes that
\begin{equation*}
\|u^N(t) - u(t)\|_{L^2_x(\T)} \lesssim_{\|u_0\|_{H^s}} \|q(t)-q^N(t)\|_{L^2_t([0,T])}.
\end{equation*}
Here we used Plancherel theorem and H\"older's inequality as well as the bound of $\|q\|_{L^\infty_t([0,T])}\lesssim \|u_0\|_{H^s_x}$. So it suffices to show that
\begin{equation*}
\|q-q^N\|_{L^2_t({[0,T]})} \to 0.
\end{equation*}
Fix $\tau > 0$ sufficiently small to be determined later. Recalling Lemma \ref{convolution_lemma}, one has
\begin{align}
    \|q-q^N\|_{L^2_t([0,\tau])} &\leq \|S(t)(u_0-u_0^N)(0)\|_{L^2_t([0,\tau])}\nonumber\\
    &\qquad+  C\tau^{\frac{1}{2}}\left(\|q\|_{L^\infty_t ([0,\tau])} + \|q^N\|_{L^\infty_t ([0,\tau])}\right)^2  \|q - q^N\|_{L^2_t([0,\tau])}.
    \label{charge_convergence_1}
\end{align}
Here we have multiplied by an appropriate smooth cut-off function and used H\"older's inequality. For the first term in \eqref{charge_convergence_1}, we have
\begin{equation*}
\|S(t)(u_0-u_0^N)(0)\|_{L^2_t([0,\tau])} \leq \tau^{\frac{1}{2}} \|S(t)(u-u_0^N)(0)\|_{L^\infty_t ([0,\tau])} \lesssim \tau^{\frac{1}{2}}\|u_0-u^N_0\|_{H^s_x(\T)}
\end{equation*}
by the Sobolev embedding theorem and the unitarity of $S(t)$ on Sobolev spaces. For the second term, we recall that $q(t),q^N(t)$ are uniformly bounded in $t$ and $N$ by construction. Therefore, we can take $\tau$ sufficiently small to obtain
\begin{equation*}
\|q-q^N\|_{L^2_t([0,\tau])} \leq \tau^{\frac{1}{2}}\|u_0-u^N_0\|_{H^s(\T)} \to 0.
\end{equation*}
Since the $\tau$ depends only on things which are uniformly bounded on $[0,T]$, we can iterate this argument up to the local time of existence.
\end{proof}

\section{Open Problems}
\label{Section:open}
We briefly remark on a number of questions left open by our work, and possible future directions of research.
\begin{itemize}
\item From the perspective of mathematical physics, one would be interested in the case of mixed nonlinearities. For example, the case of $\delta|u|^2u + |u|^2u$. In this setting, one is able to construct globally bounded $H^1$ solutions to the corresponding smoothed NLS. However, the $|u|^2u$ term breaks the Volterra integral equation structure, making the proof of uniqueness more challenging.
\item The problem of well-posedness below the endpoint of the Sobolev embedding theorem is interesting. The problem is mass-critical, and we plan to study the problem of local well-posedness in $H^s(\T)$, $ s > 0$ in future work.
\item We have not addressed the problem of persistence of regularity for $u_0 \in H^s(\T)$, for $ s > 1$. Indeed, in analogy to the results in \cite{AT01}, one expects the solution to lie in $L^\infty ([0,T];H^{s}(\T))$ if $s < \frac{3}{2}$. We plan to address this in a future work.
\item There is also the question of the problem when $X$ is $\T^2$ or $\T^3$. In this setting, there are a number of additional challenges. Indeed, the proof of the concentrated limit of \eqref{smoothed_NLS_introduction} for $\R^3$ was only partially solved in \cite{CFNT17} for shrinking potentials. There is no known solution for $X=\R^2$. This is partially because of the larger domain mentioned in Remark \ref{higher_dimensional_remark}. 
\end{itemize}

\appendix

\section{Proof of Proposition \ref{sobolev_compact_proposition}}
\label{SEC:proposition}
In this section, we prove Proposition \ref{sobolev_compact_proposition}. We first prove the following preliminary result.
\begin{lemma}
\label{periodisation_lemma}
Let $s > 0$ and suppose that $f \in {H^s(\R)}$ is a compactly supported function. Define the periodisation of $f$ as
\begin{equation*}
\mathrm{Per}f(t) \equiv g(t) := \sum_{n \in \Z} f(t + 2\pi n).
\end{equation*}
Then $\|\mathrm{Per}f\|_{H^s(\T)} \lesssim \|f\|_{H^s(\R)}$.
\end{lemma}
\begin{proof}
Since $f$ is compactly supported, we know that the sum defining $g$ is a finite sum. Therefore
\begin{equation}
\label{periodic_extension_definition}
    g(t) = \sum_{|n| \leq M} f(t+ 2 \pi n).
\end{equation}
Suppose that $s = m$ is a positive integer and take an integer $0 \leq j \leq m$. Then
\begin{equation*}
D^j g(t) = \sum_{|n| \leq M} D^j f(t+2\pi n).
\end{equation*}
By Cauchy--Schwarz inequality, one has
\begin{equation*}
|D^j g (t)|^2 \lesssim M \sum_{|n| \leq M} |D^j f(t+ 2\pi n)|^2.
\end{equation*}
Summing over $j$, it follows that
\begin{equation*}
\|g\|_{H^m(\T)} \lesssim_M \|f\|_{H^m(\R)}.
\end{equation*}
So $\mathrm{Per}f$ is bounded on $H^m(\T)$ for $f$ with fixed compact support. The result for general $s > 0$ follows from interpolation.
\end{proof}
We can now prove Proposition \ref{sobolev_compact_proposition}.
\begin{proof}[Proof of Proposition \ref{sobolev_compact_proposition}]
{Let $g$ be the periodic extension of $f$. Since $f$ is compactly supported, we let $M$ be a finite integer such that \eqref{periodic_extension_definition} holds.} One computes that
\begin{equation*}
\hat{g}(\omega) = \frac{1}{2\pi} \int_0^{2\pi} g(t) \, \e^{-\ii \omega t} \, \dd t = \frac{1}{2\pi} \int_0^{2\pi} \sum_{|n| \leq M} f(t + 2\pi n) \, \e^{-\ii \omega t} \, \dd t.
\end{equation*}
Changing variables $t' = t + 2\pi n$ in each term of the (finite) sum, we obtain
\begin{equation*}
\hat{g}(\omega) = \frac{1}{2\pi}\sum_{|n| \leq M} \int_{2 \pi n}^{2 \pi (n+ 1)} f(t') \, \e^{-\ii \omega t'} \, \dd t' \sim_M \hat{f}(\omega).
\end{equation*}
So
\begin{equation*}
\sum_{\omega \in \Z} \langle \omega \rangle^{2s} |\hat{f}(\omega)|^2 {\sim_{M} }\|g\|^2_{H^s(\T)} \lesssim \|f\|^2_{H^s(\R)},
\end{equation*}
where we use the assumption that $f$ is compactly supported to apply Lemma \ref{periodisation_lemma}.
\end{proof}

\subsection*{Acknowledgements}
We thank Zied Ammari, Nicolas Camps, Kihyun Kim, and Tristan Robert for helpful discussions at various points, and Alessandro Teta for constructive suggestions. The authors also thank an anonymous reviewer for their comments on a previous version of this manuscript.
J.L. is partially supported by the Swiss National Science Foundation through the NCCR SwissMAP and the SNSF Eccellenza project PCEFP\_181153, by the Swiss State Secretariat for Research and Innovation through the project P.530.1016 (AEQUA), by Global - Learning \& Academic research institution for Master's$\cdot$PhD students, and Postdocs (G-LAMP) Program of the National Research Foundation of Korea (NRF) grant funded by the Ministry of Education (No.~RS-2025-25442355), and by a grant from Kyung Hee University in 2026 (KHU-20262263). A.R. acknowledges partial funding by the ANR-DFG project (ANR-22-CE92-0013, DFG PE 3245/3-1 and BA 1477/15-1), and by the Italian Ministry of University and Research (MUR) through
the PRIN 2022 grant ``OpeN and Effective quantum Systems (ONES)''. 
\bigskip

\bibliographystyle{abbrv}
\bibliography{refss}

\end{document}